\documentclass{amsart}

\newtheorem{theorem}{Theorem}[section]
\newtheorem{lemma}[theorem]{Lemma}
\newtheorem{corollary}[theorem]{Corollary}
\newtheorem{proposition}[theorem]{Proposition}

\theoremstyle{definition}
\newtheorem{definition}[theorem]{Definition}
\newtheorem{example}[theorem]{Example}

\newtheorem{question}[theorem]{Question}

\theoremstyle{remark}
\newtheorem{remark}[theorem]{Remark}

\renewcommand\k{{\mathrm{k}}}
\newcommand\Bir{{\mathrm{Bir}}}
\newcommand\dom{\mathrm{dom}}

\newcommand\car{{\mathrm{char}}}

\newcommand\Aut{{\mathrm{Aut}}}
\newcommand\Hilb{{\mathrm{Hilb}}}
\newcommand\p{{\mathbb P}}
\newcommand\F{{\mathbb F}}

\newcommand\A{{\mathbb A}}
\newcommand\id{{\mathrm{id}}}

\numberwithin{equation}{section}

\begin{document}

\title{Algebraic structures of groups of birational transformations}

\author{J\'er\'emy Blanc}
\address{J\'er\'emy Blanc, Universit\"{a}t Basel, Departement Mathematik und Informatik, Spiegelgasse $1$, CH-$4051$ Basel, Switzerland.}
\email{jeremy.blanc@unibas.ch}
\thanks{The author gratefully acknowledges support by the Swiss National Science Foundation Grants  "Birational Geometry" PP00P2\_153026 and "Algebraic subgroups of the Cremona groups" 200021\_159921.}



\begin{abstract}A priori, the set of birational transformations of an algebraic variety is just a group. We survey the possible algebraic structures that we may add to it, using in particular  parametrised family of birational transformations.\end{abstract}

\maketitle

\section{Introduction}
Let $X$ be an algebraic variety defined over an algebraically closed field~$\k$. We denote by $\Bir(X)$ the group of birational transformations of $X$, and by $\Aut(X)$ its subgroup of automorphisms (biregular morphisms).

If $X$ is projective, it is known that $\Aut(X)$ has a natural structure of group scheme, maybe with infinitely many components (\cite{Mat58}, see also \cite{MO67,Han87}). In particular, it is a scheme of finite dimension.

This is false in general for $\Bir(X)$, which can be much larger. In this note, we give a survey on the following question:
\begin{center}
\textit{What kind of algebraic structure can we put on $\Bir(X)$?}\end{center}

\bigskip

As usual in algebraic geometry, even if one does not know the structure of $\Bir(X)$, one can define what is a morphism $A\to \Bir(X)$, where $A$ is an algebraic variety, or more generally a locally noetherian scheme (see~\S\ref{Sec:Func}). This corresponds to a functor 
\[\left(\text{locally noetherian schemes}\right)\to \left(\text{Sets}\right),\]
introduced by M.~Demazure \cite{De}, which is unfortunately not representable by a scheme, or more generally an ind-scheme, as we explain in \S\ref{Sec:Ind}. The functor is representable for $\Aut(X)$, if $X$ is projective, and gives the classical group scheme structure explained before. It is also representable if $X$ is affine, but by an ind-algebraic group (see \S\ref{Sec:Ind}). 

Even if we do not know what kind of structure one can put on $\Bir(X)$, the morphisms introduced define a Zariski topology on $\Bir(X)$, as explained by J.-P.~Serre in \cite{Se}. We recall this topology in $\S\ref{SubSec:ZarTop}$, and describe some of its properties. We then finish Section~\ref{Sec:Structure} by recalling what is usually called algebraic subgroup of $\Bir(X)$, and by explaining the relation with the topology and the functors/morphisms defined.

Section~\ref{Sec:Flat} consists in looking at a sub-functor of the above one, introduced by M.~Hanamura in \cite{Han87}. It corresponds to flat families of birational transformations, and has the advantage of being representable by a scheme (\S\ref{Sec:FlatDef}). The structure is compatible with the composition and behaves quite well if the variety $X$ is not uniruled (\S\ref{Sec:ReprFlat}). This is however not the case if $X$ is a general algebraic variety. We briefly describe the case where $X=\p^n$ in $\S\ref{Sec:BirPnflat}$.\\

The author thanks Michel Brion and Jean-Philippe Furter for interesting discussions during the preparation of the article.
\section{Structures given by families of transformations}\label{Sec:Structure}
\subsection{Functors $\Bir_X$ and $\Aut_X$}\label{Sec:Func}
In \cite{De}, M.~Demazure introduced the following functor (that he called $\mathrm{Psaut}$, for pseudo-automorphisms, the name he gave to birational transformations):
\begin{definition}\label{Defi:BirX}
Let $X$ be an irreducible algebraic variety and $A$ be a locally noetherian scheme. We define
\[\begin{array}{lll}
\Bir_X(A)&=&\left\{\begin{array}{l}
A\text{-birational transformations of }A\times X\text{ inducing an}\\
\text{ isomorphism }U\to V,\text{ where }U,V\text{ are open subsets}\\
\text{ of }A\times X,\text{ whose projections on }A\text{ are surjective }\end{array}\right\},\\
\Aut_X(A)&=&\left\{\begin{array}{l}
A\text{-automorphisms of }A\times X\end{array}\right\}=\Bir_X(A)\cap \Aut(A\times X).\end{array}\]
\end{definition}
The above families were also introduced and studied before in \cite{Ram64}, at least for automorphisms.
Definition~\ref{Defi:BirX} implicitly gives rise to the following notion of families, or morphisms $A\to \Bir(X)$ (as in \cite{Se,Bl10,BF13,PanRit}):
\begin{definition}\label{Defi:Morphisms}
Taking $A,X$ as above, an element $f\in \Bir_X(A)$ and a $\k$-point $a\in A(\k)$, we obtain an element $f_a\in \Bir(X)$ given by $x\dasharrow p_2(f(a,x))$, where $p_2\colon A\times X\to X$ is the second projection. 

The map $a\mapsto f_a$ represents a map from $A$ $($more precisely from the $A(k)$-points of $A)$ to $\Bir(X)$, and will be called a \emph{morphism} from $A$ to $\Bir(X)$.
\end{definition}
\begin{remark}
We can similarly define morphisms $A\to \Aut(X)$, and observe that these are exactly the morphisms $A\to \Bir(X)$ having image in $\Aut(X)$.
\end{remark}
\begin{remark}
If $X,Y$ are two irreducible algebraic varieties and $\psi\colon X\dasharrow Y$ is a birational map, the two functors $\Bir_X$ and $\Bir_Y$ are isomorphic, via $\psi$. In other words, morphisms $A\to \Bir(X)$ corresponds, via $\psi$, to morphisms $A\to \Bir(Y)$. 

If $\psi$ is moreover an isomorphism, then it also induces an isomorphism between the two functions $\Aut_X$ and $\Aut_Y$. Equivalently, morphisms $A\to \Aut(X)$ corresponds, via $\psi$, to morphisms $A\to \Aut(Y)$. 
\end{remark}

As we will see, the functor $A\to \Bir_X(A)$ is not representable by a scheme, if $X$ is a general algebraic variety (for example if $X=\p^2$).\\

Firstly, taking $X=\p^2$, one can construct very large families:
\begin{example}
For each $m\ge 1$, the following $\A^m$-birational map of $\A^m\times \p^2$
\[\begin{array}{rcl}
\A^m\times \p^2&\dasharrow &\A^m\times \p^2\\
\left((a_1,\dots,a_m), [x:y:z]\right) &\mapsto & \left((a_1,\dots,a_m), \left[xz^{m-1}:yz^{m-1}+\sum\limits_{i=1}^m a_i x^iz^{m-i}:z^{m}\right]\right) \end{array}\]

which restricts, on the open subset where $z=1$, to the automorphism
\[\begin{array}{rcl}
\A^m\times \A^2&\to &\A^m\times \A^2\\
\left((a_1,\dots,a_m), (x,y)\right) &\mapsto & \left((a_1,\dots,a_m), \left(x,y+\sum\limits_{i=1}^m a_i x^i\right)\right) \end{array}\]
yields injective morphisms $\A^m\to \Bir(\p^2)$ and $\A^m\to \Aut(\A^2)$, whose image contains the identity.
\end{example}
Of course, the same kind of example generalises to $\p^n$ and $\A^n$ for any $n\ge 2$. It shows that neither $\Bir(\p^2)$ nor $\Aut(\A^2)$ can be endowed with the structure of a locally noetherian scheme, compatible with the above families / morphisms, or equivalently says that the functor $\Bir_{\p^2}$ and $\Aut_{\A^2}$ are not representable by a locally noetherian scheme. 
\subsection{Ind-varieties and ind-groups}\label{Sec:Ind}
One way to avoid the problem of noetherianity consists of studying ind-schemes, which are inductive limits of locally noetherian schemes. One of the first articles in this direction is \cite{Sha1}, which introduces the notion of "infinite dimensional algebraic varieties", or simply ind-variety, as given by a formal inductive limit of  closed embeddings of algebraic varieties $X_i\hookrightarrow X_{i+1}$. 
\begin{definition}
An \emph{ind-scheme} (resp.~\emph{ind-variety}) is given by a countable union $(X_i)_{i\in\mathbb{N}}$ of schemes (resp.~algebraic varieties) together with closed embeddings $X_i\hookrightarrow X_{i+1}$.

A \emph{morphism} between two ind-schemes $(X_i)_{i\in\mathbb{N}}$ and $(Y_i)_{i\in\mathbb{N}}$ is given by a collection of morphisms $\rho_i \colon X_i \to Y_{j_i}$, where $\{j_i\}_{i\in \mathbb{N}}$ is a sequence of indices, which is compatible with the inclusions.
\end{definition}

The aim of this construction was to study the groups $\Aut(\A^n)$, which are ind-algebraic varieties, as shown in \cite{Sha2}. The group structure being compatible, the groups $\Aut(\A^n)$ are then shown to be ind-algebraic groups (see again \cite{Sha2}), even if the $X_i$ are not subgroups. One can moreover observe that this structure gives the representability of the functor $\Aut_{\A^n}$ by an ind-algebraic group \cite[Lemma 2.7]{Bl15}.

More generally, for any affine algebraic variety $X$, the group $\Aut(X)$ can be seen as an ind-group \cite{KM2005}. This again gives  the representability of the functor $\Aut_{X}$ by an ind-algebraic group (see \cite[Theorem 3.3.3]{KM2005} or \cite{FurKra}).

\bigskip

After having introduced this new category, the natural question to ask is wether the functor $\Bir_X$ can always be represented by an ind-scheme. This what I.R. Shafarevich asked in  \cite[\S 3]{Sha1}:
"Can one introduce a universal structure of an infinite-dimensional group
in the group of all automorphisms (resp. all birational automorphisms)
of arbitrary algebraic variety?"\\

The answer, given in \cite{BF13}, is negative, and can again been shown explicitly for the case of $\p^2$. The problem does not come from the infinite dimension but from the degenerations of birational maps of high degree to maps of smaller degree. Let us give the following example (\cite[Example 3.1]{BF13}):
\begin{example}\label{TheExample}
Let $\hat{V}=\p^2\backslash\{[0:1:0],[0:0:1]\}$, let $\rho\colon \hat{V}\to \Bir(\p^2)$ be the morphism given by 
\[\begin{array}{rcl}
\hat{V}\times \p^2&\dasharrow &\hat{V}\times \p^2\\
\left([a:b:c], [x:y:z]\right) &\mapsto & \left([a:b:c], \left[x(ay+bz):y(ay+cz):z(ay+cz)\right]\right) \end{array}\]
and define $V\subseteq \Bir(\p^2)$ to be the image of $\rho$.
The map ${\rho}\colon\hat{V}\to V$ sends the line $L\subseteq \hat{V}$ corresponding to $b=c$ to the identity,
and induces a bijection $\hat{V}\backslash L \to V\backslash \{\id\}$.
\end{example}
\begin{remark}
The above map corresponds, on the affine plane where $z=1$, to
$$\begin{array}{rccl} \hat{V}\times \A^2&\dasharrow& \hat{V}\times \A^2,\\
\left([a:b:c],(x,y)\right)&\mapsto& \left([a:b:c],(x\cdot \frac{ay+b}{ay+c},y)\right).\end{array}$$
\end{remark}

With this example, one can see that the structure of $V\subset\Bir(\p^2)$ should be the quotient of $\hat{V}\to V$, i.e.~the the quotient of $V$ modulo the equivalence relation that identifies all points of $L$ \cite[Lemma 3.3]{BF13}. As this line is equivalent to any other general line, the structure obtained is not the one of an algebraic variety, or even of an algebraic space. It shows that $\Bir_{\p^2}$ is not representable by an ind-variety, or even an ind-algebraic space or ind-algebraic stack \cite[Proposition 3.4]{BF13}. We summarise it here:
\begin{theorem}$($\cite[Theorem 1]{BF13}$)$
For each $n\ge 2$, there is no structure of ind-algebraic variety $($or algebraic variety$)$ on $\Bir(\p^n)$, such that morphisms $A\to \Bir(\p^n)$ correspond to the morphisms of ind-algebraic varieties $A\to \Bir(\p^n)$.
\end{theorem}

Despite of this, it could be interesting to study equivalence classes on algebraic varieties. If the relation is closed and \'etale, one obtains an algebraic space \cite[Definition 2.3]{Artin}. One could then seek for generalisations of this, by admitting non-\'etale equivalence relations, like the one induced by the above example. It would however introduce some pathologies: the local ring at the special point of $\mathrm{id}\in V$ corresponds to functions defined on a open set of $\hat{V}$ containing $L$ and would then be the ring of constant functions. This gives rise to the following question:
\begin{question}
Can we enlarge the category of ind-scheme to a "not too nasty" category in order to be able to represent the functor $\Bir_{\p^2}$ ? $($or $\Bir_X$ in general$)$ ?
\end{question}

Another question would be to determine the varieties $X$ for which $\Bir_X$ can be represented by an ind-scheme. In particular, the following question arises:
\begin{question}
Is there an algebraic variety $X$ such that $\Bir_X$ can be represented by an ind-scheme, but not by a group scheme?
\end{question}

\subsection{Group structure and Zariski topology on $\Bir(X)$}\label{SubSec:ZarTop}
Note that the inverse map yields an isomorphism of functors from $\Bir_X$ to itself. Similarly, we can define a functor $\Bir_X\times \Bir_X$, in the same way as for $\Bir_X$, and then observe that the composition is a morphism of functors. The notion of families given by $\Bir_X$ is then compatible with the group structure.

Even if $\Bir_X$ is not representable, we can define a topology on the group $\Bir(X)$, given by this functor. This topology was called \emph{Zariski topology} by J.-P. Serre in \cite{Se}:
\begin{definition}  \label{defi: Zariski topology}
Let $X$ be an algebraic variety. A subset~$F\subseteq \Bir(X)$ is \emph{closed in the Zariski topology}
if for any algebraic variety~$A$ $($or more generally any locally noetherian algebraic scheme$)$ and any morphism~$A\to \Bir(X)$ the preimage of~$F$ is closed.
\end{definition}
In the case where $\Bir_X$ is represented by an algebraic group, then the above topology is compatible with the Zariski topology of the algebraic group. Moreover, even if $\Bir(X)$ is not an algebraic group, then its topology and group structure behave not so far from algebraic groups. For instance, we can define the Zariski topology on $\Bir(X)\times \Bir(X)$, using morphisms as above, and check that the composition law yields a continuous map $\Bir(X)\times \Bir(X)\to \Bir(X)$. Moreover, the map sending an element on its inverse is a homeomorphism $\Bir(X)\to \Bir(X)$. Similarly, taking powers, left and right-multiplications and conjugation are homeomorphisms (see for example \cite[Lemma 2.3]{Bl14}). 
Using such properties, one can see for instance that the closure of a subgroup is again a subgroup, and that the closure of an abelian subgroup (for example a cyclic group) is abelian.\\

For $n\ge 2$, the Zariski topology of $\Bir(\p^n)$ is not the one of any algebraic variety, or even ind-variety \cite[Theorem 2]{BF13}. The obstruction follows from the bad topology of the set $V$ constructed in Example~\ref{TheExample}: it contains a point where all closed subsets of positive dimension pass through. 

However, we can describe the topology of $\Bir(\p^n)$, using maps of low degree.
\begin{definition}\label{Def:Degree}
For each $\varphi\in \Bir(\p^n)$, the \emph{degree} of $\varphi$ is the degree $\deg(\varphi)$ of the pull-back of a general hyperplane. Equivalently, it is the degree of the polynomial that define $\varphi$, when these are taken without common factor.

We define by $\Bir(\p^n)_d$ (respectively by $\Bir(\p^n)_{\le d}$)  the set of elements of $\Bir(\p^n)$ of degree exactly $d$ (respectively of degree $\le d)$.
\end{definition}
\begin{remark}
We have $\Bir(\p^n)_1=\Bir(\p^n)_{\le 1}=\Aut(\p^n)$.
\end{remark}
We can first remark that $\Bir(\p^n)_{\le d}$ is closed in $\Bir(\p^n)$ for each $d$ \cite[Corollary 2.8]{BF13}. This is the semi-continuity of the degree, which was also proved in \cite[Lemma 4.1]{Xie15} for arbitrary surfaces. Then, the topology of $\Bir(\p^n)$ can be deduced from its subsets of bounded degree:

\begin{lemma}\cite[Proposition 2.10]{BF13} The topology of $\Bir(\p^n)$ is the inductive limit topology given by the Zariski topologies of $\Bir(\p^n)_{\le d}$, $d\in \mathbb{N}$, which are the quotient topology of $\pi_d\colon H_d\to\Bir(\p^n)_{\le d}$, where $H_d$ is an algebraic variety, endowed with its Zariski topology.\end{lemma}
The algebraic varieties $H_d$ are given by $(n+1)$-uples of homogeneous polynomials of degree $d$ inducing birational maps. The map $\pi_d\colon H_d\to\Bir(\p^n)_{\le d}$ restricts then to a bijection on $(\pi_d)^{-1}(\Bir(\p^n)_{d})$, but not on maps of smaller degree, that can be represented in many different ways in $H_d$, by multiplying each coordinate by the same factor. These distinct possible factors are responsible of the fact that the Zariski topology of $\Bir(\p^n)_{\le d}$ is not the one of an algebraic variety.\\

Note that $\Bir(\p^n)$ is connected for each $n$ \cite{Bl10}, and that $\Bir(\p^2)_{d}$  is connected for $d\le 6$ \cite{BCM}. Moreover, $\Bir(\p^2)$ does not contain any closed normal subgroup \cite{Bl10}, even if it is not simple, viewed as an abstract group \cite{CL13}.\\

The Zariski topology of $\Bir(X)$, for an arbitrary algebraic variety $X$, it still not well understood.

\subsection{Algebraic subgroups}\label{AlgSubgroups}
Studying biregular actions of algebraic groups on algebraic varieties is a very classical subject of algebraic geometry. More generally, one can study rational actions of algebraic groups. This was done for example in \cite{Wei55,Ros56}. Using the notion of morphism $A\to \Bir(X)$ of Definition~\ref{Defi:Morphisms}, the algebraic actions and algebraic subgroups of $\Bir(X)$ can be naturally defined:
\begin{definition}
Let $X$ be an irreducible algebraic variety and $G$ be an algebraic group. A \emph{birational group action} (respectively \emph{biregular group action}) of $G$ on $X$ is a morphism $G\to \Bir(X)$ (respectively $G\to \Aut(X)$) which is also a group homomorphism. The image of this morphism is a subgroup of $\Bir(X)$ (respectively of $\Aut(X)$) which is called \emph{algebraic subgroup}.
\end{definition}

Note that any birational map $X\dasharrow Y$ conjugate birational group actions on $X$ to birational group actions on $Y$. This allows sometimes to obtain biregular group actions: 

\begin{theorem}$($\cite[Theorem page 375]{Wei55}, \cite[Theorem 1]{Ros56}$)$
Let $X$ be an irreducible algebraic variety, $G$ be an algebraic group and $G\to \Bir(X)$ a birational group action.
Then, there exists a birational map $X\dasharrow Y$, where $Y$ is another algebraic variety, that conjugates this action to a biregular group action.
\end{theorem}
In this theorem, we can moreover assume $Y$ to be projective, using equivariant completions (see \cite{Sum74}). In particular, studying connected rational algebraic actions on a variety $X$ amounts to study the connected components of the group scheme $\Aut(Y)$, where $Y$ is a projective algebraic variety $Y$ birational to $X$. This allows for instance to show that maximal connected subgroups of $\Bir(\p^2)$ are $\Aut(\p^2)$, $\Aut(\p^1\times \p^1)^\circ$, $\Aut(\F_n)$, $n\ge 2$.\\

One can characterise the algebraic subgroups of $\Bir(\p^n)$ only using the Zariski topology defined in $\S\ref{SubSec:ZarTop}$. These are the closed subgroups of bounded degree:
\begin{theorem}$($\cite[Corollary 2.18, Proposition 2.19]{BF13}$)$\label{Thm:AlgGrCremona}
\begin{enumerate}
\item
Every algebraic subgroup of $\Bir(\p^n)$ is closed $($for the Zariski topology$)$ and of bounded degree.
\item
For each closed algebraic subgroup  $G\subset \Bir(\p^n)$ of bounded degree, there is a unique algebraic group structure on $G$, compatible with the group structure of $\Bir(\p^n)$, and such that morphisms $A\to \Bir(\p^n)$ having image in $G$ correspond to morphisms of algebraic varieties $A\to G$.
\end{enumerate}
\end{theorem}

There is also a characterisation of connected algebraic subgroups of $\Bir(X)$, for any irreducible algebraic variety $X$:

\begin{theorem}$($\cite{Ram64}$)$
Let $X$ be an irreducible algebraic variety and $G\subset \Aut(X)$ be a subgroup having the following properties:
\begin{enumerate}
\item
$($connectedness$)$ For any $f\in G$, there is a morphism $A\to \Aut(X)$, where $A$ is an irreducible algebraic variety, whose image contains $f$ and the identity.
\item
$($bounded dimension$)$ There is an integer $d$ such that for any injective morphisms $A\to \Aut(X)$ having image contained in $G$, we have $\dim A\le d$.
\end{enumerate}
Then, there is a unique structure of algebraic group on $G$, compatible with the group structure of $\Aut(X)$, such that morphisms $A\to \Aut(X)$ having image into $G$ correspond to morphism of algebraic varieties $A\to G$.
\end{theorem}

This nice result gives in particular an algebraic group structure on any algebraic subgroup of $\Bir(X)$ and implies that the Zariski topology induced by $\Bir(X)$ is the Zariski topology of the algebraic group obtained. It also seems that every algebraic subgroup of $\Bir(X)$ is closed, as stated in \cite{Pop13some,Pop13tori}. The case of $\p^n$ is given by Theorem~\ref{Thm:AlgGrCremona} above but we did not find a proof of this statement for a general algebraic variety $X$.
\section{Flat families and scheme structure}\label{Sec:Flat}
\subsection{The functor $\Bir_X^{\mathrm{flat}}$}\label{Sec:FlatDef}
Another way of studying (bi)-rational maps between projective algebraic varieties consists of studying graphs. This was the viewpoint of \cite{Han87}. Let us recall the following basic notions:
\begin{definition}
Let $X,Y$ be irreducible algebraic varieties and  $f\colon X\dasharrow Y$ a rational map. The \emph{graph of $f$} is denoted $\Gamma_f$ and is the closure of \[\{(x,f(x))\mid x\in \dom(f)\}\] in $X\times Y$.
\end{definition}
\begin{lemma}\label{Lemm:Graphs}
Let $X$ be a locally noetherian scheme, and denote by $\pi_i\colon X\times X\to X$ the $i$-th projection, for $i=1,2$. Then, the following maps are bijective:
\[\begin{array}{ccc}
\Bir(X)&\to&\left\{\begin{array}{l}
\text{irreducible closed subsets }Y\subset X\times X\\
\text{such that }\pi_i\colon Y\to X,\text{ is a birational}\\
\text{morphism, for }i=1,2.\end{array}\right\},\\
f&\mapsto &\Gamma_f.\\
\Aut(X)&\to&\left\{\begin{array}{l}
\text{irreducible closed subsets }Y\subset X\times X\\
\text{such that }\pi_i\colon Y\to X,\text{ is an}\\
\text{isomorphism, for }i=1,2.\end{array}\right\},\\
f&\mapsto &\Gamma_f.\end{array}\]
\end{lemma}
Applying this to $\Bir_X(A)$ (see Definition~\ref{Defi:BirX}), we obtain the following:
\begin{lemma}\label{Lem:BirXAgraph}
Let $X$ be an irreducible algebraic variety and $A$ be a locally noetherian scheme. We have a bijection
\[\begin{array}{ccc}
\Bir_X(A)&\to&\left\{\begin{array}{l}
\text{irreducible closed subsets }Y\subset A\times X\times X\\
\text{admitting a dense open subset }W\subset Y\\
\text{such that the  projection}  W\to \text{is surjective, and}\\
\text{such that the two projections }  A\times X\times X\to A\times X\\
\text{restrict to open immersions }W \to A\times X.\end{array}\right\},\\
f&\mapsto &\text{closure of }\{(a,x,\pi_2(f(a,x)))\mid (a,x)\in \dom(f)\}.\end{array}\]
\end{lemma}
\begin{proof}
The set $\Bir_X(A)$ corresponds to a subset of $\Bir(A\times X)$. By Lemma~\ref{Lemm:Graphs}, this latter is in bijection with irreducible closed subsets $Y\subset (A\times X )\times (A\times X)$ such that $\pi_i\colon Y\to A\times X$ is a birational morphism, for $i=1,2$. Moreover, $f\in \Bir(A\times X)$ is sent onto the closure of $\{((a,x),f(a,x))\mid (a,x)\in\dom(f)\}$.

As $\Bir_X(A)$ only consists of $A$-birational maps, we can forget one copy of $A$ and obtain the closure of $\{((a,x),\pi_2(f(a,x)))\mid (a,x)\in\dom(f)\}$ in $A\times X\times X$, which is an irreducible closed subset $Y\subset A\times X\times X$ such that the two projections to $A\times X$ are birational. As before, every such subset provides in turn an $A$-birational map of $A\times X$.

A $A$-birational map $f$ yields an element of $\Bir_X(A)$ if and only if there exist two open subsets $U,V\subset A\times X$, whose projections on $A$ are surjective and such that the map $f$ induces an isomorphism $U\to V$. Denoting by $\mu_1,\mu_2\colon A\times X\times X\to A\times X$ the two projections, the set $W=(\mu_1)^{-1}(U)\cap (\mu_2)^{-2}(V)$ is an open subset of $Y$, and the two projections give isomorphisms $\mu_1\colon W\to U$ and $\mu_2\colon W\to V$. Conversely, the existence of $W$ and of two open embeddings to $A\times X$ yields $U$ and $V$, and thus an element of $\Bir_X(A)$.
\end{proof}

Using these bijections, one can define the subfunctor $\Bir_X^{\mathrm{flat}}$ of $\Bir_X$, corresponding to flat families: 

\begin{definition}(\cite[Definition 2.1]{Han87})
Let $X$ be a projective algebraic variety and $A$ a locally noetherian scheme.
A \emph{flat family of birational transformations} (resp.~\emph{of automorphisms}) of $X$ over $A$ is a closed subscheme $Y\subset A\times X\times X$, flat over $A$, such that for each $a\in A$, the fibre $Y_a$ is the graph of an element of $\Bir_X(a)$ (respectively of $\Aut_X(a)$).
\end{definition}

\begin{definition}
Let $X$ be an algebraic variety and $A$ a locally noetherian scheme. We define $\Bir_X^{\mathrm{flat}}(A)\subset \Bir_X(A)$ as the set of elements $f\in \Bir_X(A)$ such that the corresponding graph in $A\times X\times X$ (see Lemma~\ref{Lem:BirXAgraph}) is a flat family of birational transformations.

We similarly define $\Aut_X^{\mathrm{flat}}(A)=\Bir_X^{\mathrm{flat}}(A)\cap \Aut_X(A)$.
 
\end{definition}
\subsection{Representability of $\Bir_{\p^n}^{\mathrm{flat}}$}\label{Sec:ReprFlat}
As M.~Hanamura explains in \cite[Remark 2.10]{Han87}, the advantage of $\Bir_X^{\mathrm{flat}}$ over $\Bir_X$ is that it is representable.

\begin{remark}
 Recall that $\Hilb(X\times X)$ is an algebraic scheme (locally noetherian but with infinitely many components) that represents the functor $A\to \Hilb_{X\times X}(A)$, where \[\begin{array}{lll}
\Hilb_{X\times X}(A)&=&\left\{\begin{array}{l}
\text{closed subsets }Y\subset A\times X\times X
\text{ that are flat over  }A\end{array}\right\}.\end{array}\]
Hence, $\Aut_X^{\mathrm{flat}}$ and $\Bir_X^{\mathrm{flat}}$ are subfunctors of $\Hilb_{X\times X}$.
\end{remark}
\begin{proposition}[\cite{Han87}]\label{Prop:RepresentBirXFlat}
Let $X$ be an irreducible algebraic variety. For each locally noetherian scheme $A$, 
$\Aut_X^{\mathrm{flat}}(A)$ and $\Bir_X^{\mathrm{flat}}(A)$ are open subsets of $\Hilb_{X\times X}(A)$. Hence, both $\Aut_X^{\mathrm{flat}}$ and $\Bir_X^{\mathrm{flat}}$ are representable by a scheme.
\end{proposition}

However, $\Bir_X^{\mathrm{flat}}$ has some "nasty properties", as M. Hanumara explains :
"It turns out, however, that the scheme $\Bir(X)$ has some nasty properties; it is not a group scheme in general; even when $X$ and $X'$ are birationally equivalent, $\Bir(X)$ and $\Bir(X')$ may not be isomorphic." Another problem is that the composition law $\Bir(X)\times \Bir(X)\to \Bir(X)$ is not a morphism in general (see Corollary~\ref{Coro:ProductNotMorphism}). The essential reason for these "nasty properties" is that the flatness of the graphs is not invariant under birational maps $X\dasharrow Y$ and even under birational transformations of $X$. 

One example is given in \cite[(2.9)]{Han87}, comparing an abelian variety $A$  of dimension $n\ge 2$ and the blow-up $\tilde{A}\to A$ at one point. Then, $\dim \Bir^\circ(A)=n$ but $\dim \Bir^\circ(\tilde{A})=0$, hence $\Bir(A)$ and $\Bir(\tilde{A})$ are not isomorphic. Moreover, $\Bir(\tilde{A})$ is not even equi-dimensional. 
In $\S\ref{Sec:BirPnflat}$, we will describe more precisely the case of $\p^n$. 

In the case where $\car(\k)=0$ and where $X$ is a terminal model, it is however proved in \cite{Han87} that the scheme obtained has a group scheme structure, compatible with the group structure of $\Bir(X)$. This has been generalised in \cite{Han88}, in the case of non-uniruled varieties.

\begin{theorem}$($\cite[Theorem 2.1]{Han88}$)$\label{ThmHan881}
Let $X$ be a non-uniruled projective variety over an algebraically closed field $\k$ of characteristic $0$, and let us put on $\Bir(X)$ the scheme structure that represents $\Bir_X^{\mathrm{flat}}$ $($see Proposition~$\ref{Prop:RepresentBirXFlat})$. Then, the following hold:

\begin{enumerate}
\item
$\dim \Bir(X)\le \min \{\dim X, q(X)\}$, where $q(X)$ denotes the irregularity of a non-singular model of $X$.
\item
There exists a projective variety $Y$ $($which may be taken non-singular$)$, birational to $X$, such that $\Bir(Y)_{\mathrm{red}}$ has a natural structure of a group scheme, locally of finite type over $\k$.
\item
$\Bir(Y)_{\mathrm{red}}$ contains $\Aut(Y)$ as an open and closed group subscheme; $\Bir^\circ(Y)$ coincides with  $\Aut^\circ(Y)$ and is an abelian variety.
\end{enumerate}
\end{theorem}
\begin{theorem}$($\cite[Theorem 2.2]{Han88}$)$
Let $X$ and $Y$ be as in Theorem~$\ref{ThmHan881}$. Then, the following hold:

\begin{enumerate}
\item
Let $G$ be a group scheme, locally of finite type over $\k$. Then to give a birational action of $G$ on $X$ $($as in \S$\ref{AlgSubgroups})$ is equivalent to a homomorphism of group schemes $G\to \Bir(Y)_{\mathrm{red}}$.
\item
Let $Y'$ be another projective variety birational to $X$ with the property that $\Bir(Y')_{\mathrm{red}}$ is also group scheme. Then, $\Bir(Y)_{\mathrm{red}}$ and $\Bir(Y')_{\mathrm{red}}$ are isomorphic as group schemes.
\end{enumerate}
\end{theorem}

\subsection{The functors $\Bir_{\p^n}^{\mathrm{flat}}$}\label{Sec:BirPnflat}
As explained before, the functor $\Bir_{X}^{\mathrm{flat}}$ is representable by a scheme, for any algebraic variety $X$ (Proposition~\ref{Prop:RepresentBirXFlat}). Let us illustrate the structure that we obtain, in the case where $X=\p^n$. Using the notion of degree of a birational map of $\p^n$ (Definition~\ref{Def:Degree}), one can  define a subfunctor $\Bir_{\p^n}^{\mathrm{deg}}$ of $\Bir_{\p^n}$. For each $A$ we define $\Bir_X^{\mathrm{deg}}(A)\subset \Bir_X(A)$ as the elements $f\in \Bir_X(A)$ such that the corresponding morphism $A\to \Bir_X(A)$ has constant degree on connected components of $A$. Similarly, we can define $\Bir_{\p^n}^{d}$, for each integer $d$, by taking only maps of degree $d$.
\begin{lemma}Let $n\ge 2$ be an integer.
\begin{enumerate}
\item For each $d\ge 1$, the functor $\Bir_{\p^n}^{d}$ is representable by an algebraic variety. This gives to the set $\Bir(\p^n)_d$ a natural structure of algebraic variety.
\item The functor $\Bir_{\p^n}^{\mathrm{deg}}$ is representable by an algebraic scheme.
\end{enumerate}
\end{lemma}
\begin{proof}
The first part is the statement  of \cite[Proposition 2.15(b)]{BF13}. The second part follows from the first one, by taking the disjoint union of the $\Bir(\p^n)_d$.
\end{proof}
\begin{remark}Note that the structure of algebraic variety of $\Bir(\p^n)_d$ is obtained by associating to an element 
\[\begin{array}{rrcl}
f\colon& \p^n& \dasharrow &\p^n\\
&[x_0:\dots:x_n]&\mapsto &[f_0(x_0,\dots,x_n):\dots:f_n(x_0,\dots,x_n)]\end{array}\]
its coordinates $[f_0:\dots:f_n]$, that lives in the projective space parametrising the $(n+1)$-uples of polynomials of degree $d$, up to scalar multiplication (see \cite{BF13} or \cite[\S1]{BCM} for more details).
\end{remark}
The notion of degree can be generalised: we can associate to any element $f\in \Bir(\p^n)$  a sequence of integers $(d_1,\dots,d_{n-1})$ called \emph{multidegree} of $f$ in \cite{PanMult}, \cite[\S7.1.3]{DolgBook} (or characters in \cite{SemTyr}). By definition, $d_i$ is equal to the degree of $f^{-1}(H_i)$, where $H_i\subset \p^n$ is a general linear subspace of codimension~$i$. In particular, $d_1=\deg(f)$ and $d_{n-1}=\deg(f^{-1})$. Another way to see the multidegree is to observe that the graph $\Gamma_f\subset \p^n\times \p^n$ is equal, in the chow ring of $\p^n\times \p^n$, to $\sum_{i=0}^n d_i h_{n-i,i}$, where $h_{i,j}$ denotes  the class of a linear subspace $\p^i\times \p^j$, where $d_0=d_n=1$ and where $(d_1,\dots,d_{n-1})$ is the multidegree of $f$. See \cite{PanMult} or \cite[\S7.1.3]{DolgBook} for more details.
\begin{lemma}\label{Lemm:Fixeddegree}
Let $n\ge 2$ be an integer.
\begin{enumerate}
\item For each locally noetherian scheme $A$ and each $f\in \Bir_{\p^n}^{\mathrm{flat}}(A)$, the induced morphism $A\to \Bir(\p^n)$ has constant multidegree on connected components of $A$.
\item The functor $\Bir_{\p^n}^{\mathrm{flat}}$ is a subfunctor of $\Bir_{\p^n}^{\mathrm{deg}}$.  
\end{enumerate}\end{lemma}
\begin{proof}
Let $A$ be a locally noetherian scheme, and let $f\in \Bir_{\p^n}(A)$, which corresponds to a morphism $\rho_f\colon A\to \Bir(\p^n)$, and to an irreducible subset $Y$ of $A\times \p^n\times \p^n$, given by the closure of $\{(a,x,\pi_2(f(a,x)))\mid (a,x)\in \dom(f)\}$ (see Lemma~\ref{Lem:BirXAgraph}). Moreover, $\pi_2(f(a,x))=\rho_f(a)(x)$ for each $(a,x)\in \dom(f)$. 

By definition, $f\in\Bir_{\p^n}^{\mathrm{flat}}(A)$ if and only if $Y$ is flat over $A$ and the fibre of each $a\in A$ is the graph of an element of $\Bir_X(a)$.

 The flatness of $Y$ over $A$ implies that the classes of the fibres $Y_a$ of elements $a\in A$ are locally constant in the chow ring of $\p^n\times \p^n$, and thus that the multidegree of $\rho_f$ is constant on connected components of $A$. In particular, $\Bir_{\p^n}^{\mathrm{flat}}(A)\subset \Bir_{\p^n}^{\mathrm{deg}}(A)$.
\end{proof}
\begin{example}We choose $A=\A^1$ and consider the morphism $\kappa\colon A\to \Bir(\p^2)$ given by
$$\kappa(t)\colon [x:y:z]\mapsto [-xz+ty^2:yz: z^2].$$
For $t\not=0$, $\kappa(t)$ is a quadratic birational involution of $\p^2$, but $\kappa(0)$ is equal to the linear automorphism $[x:y:z]\mapsto [-x:y:z]$. As the degree drops, the corresponding family is not flat over $A=\A^1$.

We can observe this by looking at the corresponding graph: 
\[Y=\left\{
([x:y:z],[X:Y:Z],t)\in A\times \p^2 \times \p^2\left\vert \begin{array}{lll}Yz&=&Zy\\
 Xz^2&=&Z(-xz+ty^2)\\
 Xyz&=&Y(-xz+ty^2)\end{array}\right.\right\}.\]
When $t\not=0$, the fibre $Y_t$ is the graph of $\kappa(t)$, which is an irreducible surface in $\p^2\times \p^2$. When $t=0$, the fibre $Y_0$ is the union of the graph of $\kappa(0)$ and of the surface given by $z=0,Z=0$.
\end{example}
\begin{example}\label{Example:Flatstandard}We choose again $A=\A^1$ and consider the morphism $\nu\colon A\to \Bir(\p^n)$, given by
\[\nu(t)\colon [x_0:\dots:x_n]\mapsto \left[\frac{1}{x_0+tx_n}:\frac{1}{x_2}:\dots:\frac{1}{x_n}\right].\]
This morphism corresponds to the composition of the standard transformation $\nu(0)$ with a family of automorphism and is thus flat by \cite[Proposition 2.5]{Han87}. We can also verify this by looking at the corresponding graph and observing that the fibre of $t\in \A^1$ is the graph of $\nu(t)$.
\end{example}

\begin{corollary}\label{Coro:ProductNotMorphism}
Putting on $\Bir(\p^n)$ the scheme structure provided by the representability of $\Bir_{\p^n}^{\mathrm{flat}}$, the following hold:
\begin{enumerate}
\item
The set $\Bir(\p^n)_d$ is open in $\Bir(\p^n)$, for each $d$.
\item
For $n\ge 2$, the multiplication map $\Bir(\p^n)\times \Bir(\p^n)\to \Bir(\p^n)$ is not a morphism: it is not even continuous.
\end{enumerate}
\end{corollary}
\begin{proof}
The part $(1)$ follows from Lemma~\ref{Lemm:Fixeddegree}(1). 

To see $(2)$, we consider the morphism $\nu\colon \A^1\to \Bir(\p^n)$ given in Example~\ref{Example:Flatstandard}. Since the family is flat over $\A^1$, it corresponds to a morphism of schemes. We then define  
\[\begin{array}{rrcl}
\nu'\colon &\A^1& \to & \Bir(\p^n)\\
&t& \mapsto & \nu(t)\circ\nu(0)\end{array}\]
which is a morphism in the sense of Definition~\ref{Defi:Morphisms}, but not a morphism of schemes as it corresponds to an element of $\Bir_{\p^n}(\A^1)\setminus \Bir_{\p^n}^{\mathrm{flat}}(\A^1)$. Indeed, $\nu'(0)$ is the identity, which is of degree $1$, but for $t\not=0$, the element $\nu'(t)\in \Bir(\p^n)$ is the quadratic transformation
\[ \begin{array}{rcl}
[x_0:\dots:x_n]&\mapsto &\left[\frac{1}{1/x_0+t/x_n}:x_1:\dots:x_n\right]\\
&=&\left[x_0x_n:x_1(x_n+tx_0):\dots:x_n(x_n+tx_0)\right].\end{array}\]
In particular, $\nu'$ is not continuous, as $(\nu')^{-1}(\Bir(\p^n)_{1})=\{0\}$ is not open, so the composition map  $\mathrm{mult}\colon\Bir(\p^n)\times \Bir(\p^n)\to \Bir(\p^n)$  is not continuous.
\end{proof}

We finish this text by comparing the two scheme structures on $\Bir(\p^n)$ given by the functors $\Bir_{\p^n}^{\mathrm{flat}}$ and  $\Bir_{\p^n}^{\mathrm{deg}}$.
\begin{lemma} 
The functors $\Bir_{\p^n}^{\mathrm{flat}}$ and $\Bir_{\p^n}^{\mathrm{deg}}$ are not equal if $n\ge 3$.
\end{lemma}
\begin{proof}
If $n\ge 3$, we can easily find some families of constant degree but having inverse of non-constant degree. This shows that the functors $\Bir_{\p^n}^{\mathrm{flat}}$ and $\Bir_{\p^n}^{\mathrm{deg}}$ are not equal. Take for example the family of automorphisms of $\A^n$ given by
$$\begin{array}{llll}
\xi(t)\colon&(x_1,\dots,x_n)&\mapsto& (x_1+(x_2)^2,x_2+t(x_3)^2,x_3,\dots,x_n),\\
\xi(t)^{-1}\colon&(x_1,\dots,x_n)&\mapsto& (x_1-(x_2-t(x_3)^2)^2,x_2-t(x_3)^2,x_3,\dots,x_n),\end{array}
$$
and extend it to a family of birational transformations of $\p^n$. We then find $\deg(\xi(t))=2$, $\deg(\xi(t)^{-1})=4$ for each $t\not=0$, but $\deg(\xi(0))=\deg(\xi(0)^{-1})=2$.

\end{proof}

\begin{remark}
It seems to us that $\Bir_{\p^2}^{\mathrm{flat}}=\Bir_{\p^2}^{\mathrm{deg}}$. One reason for this is that the Hilbert polynomial of the graph of an element $f\in \Bir(\p^2)_d$ is $P(x)=x^2(d+1)+3x+1$ (when we view this graph in $\p^8$ via the Segre embedding $\p^2\times \p^2\to \p^8$), and is then only dependent of degree $d$.
\end{remark}

\begin{question}Is $\Bir_{\p^n}^{\mathrm{flat}}$ corresponding to algebraic families with a fixed multidegree $($on connected components$)$?\end{question}

\bibliographystyle{amsalpha}

\begin{thebibliography}{Pop2013b}


\bibitem[Art1971]{Artin}
\textsc{Michael Artin}:
\textit{Algebraic spaces.}
A James K. Whittemore Lecture in Mathematics given at Yale University, 1969. Yale Mathematical Monographs, 3. Yale University Press, New Haven, Conn.-London, 1971.

\bibitem[BCM2015]{BCM}
\textsc{Cinzia Bisi, Alberto Calabri, Massimiliano Mella}:
\textit{On plane Cremona transformations of fixed degree.}
J. Geom. Anal. {\bf 25} (2015), no. 2, 1108--1131. 


\bibitem[Bla2010]{Bl10}
\textsc{J\'er\'emy Blanc}:
\textit{Groupes de Cremona, connexit\'e et simplicit\'e.}
Ann. Sci. Ec. Norm. Sup\'er. {\bf 43} (2010), no. 2, 357-364.




\bibitem[Bla2014]{Bl14}
\textsc{J\'er\'emy Blanc}:
\textit{Algebraic elements of the Cremona groups.}
http://arxiv.org/abs/1409.1139


\bibitem[Bla2015]{Bl15}
\textsc{J\'er\'emy Blanc}:
\textit{Conjugacy classes of special automorphisms of the affine spaces.}
http://arxiv.org/abs/1412.8733



\bibitem[BF2013]{BF13}
\textsc{J\'er\'emy Blanc, Jean-Philippe Furter}:
\textit{Topologies and structures of the Cremona groups.} 
Ann. of Math. {\bf 178} (2013), no. 3, 1173--1198. 


\bibitem[CL2013]{CL13}
\textsc{Serge Cantat, St\'ephane Lamy}:
\textit{Normal subgroups in the Cremona group.}
With an appendix by Yves de Cornulier. 
Acta Math. {\bf 210} (2013), no. 1, 31--94. 


\bibitem[Dem1970]{De}  
\textsc{Michel Demazure}:
\textit{Sous-groupes alg\'ebriques de rang maximum du groupe de Cremona.}
Ann. Sci. \'Ecole Norm. Sup. {\bf 3} (1970), 507--588.

\bibitem[Dolg2012]{DolgBook}
\textsc{Igor V. Dolgachev}:
\textit{Classical algebraic geometry.}
A modern view. Cambridge University Press, Cambridge, 2012. 


\bibitem[FK2014]{FurKra}
\textsc{Jean-Philippe Furter, Hanspeter Kraft}:
\textit{On the geometry of the automorphism group of affine $n$-space.}
 manuscript in preparation

\bibitem[Han1987]{Han87}
\textsc{Masaki Hanamura}:
\textit{On the birational automorphism groups of algebraic varieties. } 
Compositio Math. {\bf 63} (1987), no. 1, 123--142. 

\bibitem[Han1988]{Han88}
\textsc{Masaki Hanamura}:
\textit{Structure of birational automorphism groups. I. Nonuniruled varieties. } 
Invent. Math. {\bf 93} (1988), no. 2, 383--403. 



\bibitem[Har1977]{Hartshorne}
\textsc{Robin Hartshorne}:
\textit{Algebraic geometry.}
Graduate Texts in Mathematics, No. 52,
Springer-Verlag, New York-Heidelberg, 1977.



 
 
\bibitem[KM2005]{KM2005}
\textsc{Tatsuji Kambayashi, Masayoshi Miyanishi}:
\textit{On two recent views of the Jacobian conjecture.}
 Affine algebraic geometry, 113--138, 
Contemp. Math., 369, Amer. Math. Soc., Providence, RI, 2005. 


\bibitem[Mat1958]{Mat58}
\textsc{Teruhisa Matsusaka}:
\textit{Polarized varieties, fields of moduli and generalized Kummer varieties of polarized varieties.}
Amer. J. Math. {\bf 80}, 45--82 (1958).


\bibitem[MO1967]{MO67}
\textsc{Hideyuki Matsumura, Frans Oort}:
\textit{Representability of group functors, and automorphisms of algebraic schemes. }
Invent. Math. {\bf 4} (1967) 1--25. 


\bibitem[Pan2000]{PanMult}
\textsc{Ivan Pan}:
\textit{Sur le multidegr\'e des transformations de Cremona}. 
C. R. Acad. Sci. Paris S\'er. I Math. {\bf 330} (2000), no. 4, 297--300. 

\bibitem[PR2013]{PanRit}
\textsc{Ivan Pan, Alvaro Rittatore}:
\textit{Some remarks about the Zariski topology of the Cremona group.}
http://arxiv.org/abs/1212.5698

\bibitem[Pop2013a]{Pop13some}
\textsc{Vladimir L. Popov}:
\textit{Some subgroups of the Cremona groups.}
Affine algebraic geometry, 213--242, World Sci. Publ., Hackensack, NJ, 2013. 


\bibitem[Pop2013b]{Pop13tori}
{\sc Vladimir L. Popov:} 
{\it Tori in the Cremona groups.} 
Izv. Math. {\bf 77} (2013), no. 4, 742--771.


 
\bibitem[Ram1964]{Ram64}
\textsc{Chakravarthi P. Ramanujam}:
\textit{A note on automorphism groups of algebraic varieties.}
Math. Ann. {\bf 156} (1964), no. 1, 25--33. 


\bibitem[Ros1956]{Ros56}
\textsc{Maxwell Rosenlicht}:
\textit{Some basic theorems on algebraic groups. } 
Amer. J. Math. {\bf 78} (1956), 401--443. 


\bibitem[ST1968]{SemTyr}
\textsc{John G. Semple, John A. Tyrrell}:
\textit{Specialization of Cremona transformations.}
Mathematika {\bf 15} 1968 171--177. 


\bibitem[Ser2010]{Se}
\textsc{Jean-Pierre Serre}:
 \textit{Le groupe de Cremona et ses sous-groupes finis.} 
S\'eminaire Bourbaki. Volume 2008/2009. Ast\'erisque No. {\bf 332} (2010), Exp. No. 1000, vii, 75--100.



\bibitem[Sha1966]{Sha1} 
\textsc{Igor R. Shafarevich}:
\textit{On some infinite-dimensional groups.}
Rend. Mat. e Appl. {\bf 25} (1966), no. 1-2, 208-212.



\bibitem[Sha1982]{Sha2}
\textsc{Igor R. Shafarevich}:
\textit{On some infinite-dimensional groups II.}
Math. USSR Izv. {\bf 18} (1982), 214-226.

\bibitem[Sum74]{Sum74}
\textsc{Hideyasu Sumihiro}
\textit{Equivariant completion.}
 J. Math. Kyoto Univ. {\bf 14} (1974), 1--28.


\bibitem[Wei1955]{Wei55}
\textsc{Andr\'e Weil}:
\textit{On algebraic groups of transformations. } 
Amer. J. Math. {\bf 77} (1955), 355--391. 


\bibitem[Xie2015]{Xie15}
\textsc{Junyi Xie}:
\textit{Periodic points of birational transformations on projective surfaces.} 
Duke Math. J. {\bf 164} (2015), no. 5, 903--932. 



\end{thebibliography}

\end{document}